\documentclass[12pt,reqno]{amsart}
\usepackage{amsmath,amssymb,hyperref,mathrsfs,graphicx}
\usepackage[utf8x]{inputenc}
\usepackage[T1]{fontenc}
\usepackage[french, english]{babel}
\usepackage{amsmath}
\usepackage{amsfonts}
\usepackage{latexsym}
\usepackage{amsthm}
\usepackage{amssymb,amscd}
\usepackage{xargs}
\usepackage{graphicx}
\usepackage{color}
\usepackage{enumitem}

\usepackage{lmodern}

\usepackage[text={5.6in, 8.3in}, centering]{geometry}
 
\usepackage{parskip}
\usepackage{microtype}

\usepackage{hyperref}

\newtheorem{thm}{Theorem}

\newtheorem{prop}{Proposition}[section]
\newtheorem{lem}[prop]{Lemma}

\newtheorem{cor}[prop]{Corollary}

\theoremstyle{definition}

\newcommand{\T}{\mathbb{T}}
\newcommand{\R}{\mathbb{R}}

\newcommand{\N}{\mathbb{N}}

\newcommand{\cS}{\mathcal{S}}

\newcommand{\cM}{\mathcal{M}}

\newcommand{\cI}{\mathcal{I}}

\newcommand{\cL}{\mathcal{L}}

\newcommand{\bV}{\mathbb{V}}

\newcommand{\cP}{\mathcal{P}}
\newcommand{\cG}{\mathcal{G}}
\newcommand{\cC}{\mathcal{C}}

\newcommand{\Sg}{\mathrm{Sg}\,}
\newcommand{\Span}{\mathrm{Span}\,}
\newcommand{\bmat}[1]{\begin{bmatrix}#1\end{bmatrix}}

\newcommand{\tcA}{\widetilde{\mathcal{A}}}
\newcommand{\tcI}{\widetilde{\mathcal{I}}}

\newcommand{\tcM}{\widetilde{\mathcal{M}}}

\newcommand{\argmin}{\mathrm{arg\,min}}
\newcommand{\Id}{\, \mathrm{I}}

\title{On the tangent cones of Aubry sets}
\author{Ke Zhang}
\address{Department of Mathematics, University of Toronto \\Ontario, Canada, M5S2E4}
\thanks{Supported by NSERC Discovery grant, reference number 436169-2013. }
\date{\today}
\begin{document}
\maketitle

\selectlanguage{french}
\begin{abstract}
	Nous montrons que le cône paratangent de l’ensemble d’Aubry du Hamiltonien de Tonelli est contenu dans un cône borné par les fibrés de Green. Notre résultat améliore un résultat précédent de M.-C. Arnaud sur les cônes tangents des ensembles d’Aubry.
\end{abstract}
\selectlanguage{english}
\begin{abstract}
	We show that the paratingent cone of the Aubry set of the Tonelli Hamiltonian is contained in a cone bounded by the Green bundles. Our result improves the earlier result of M.-C. Arnaud on tangent cones of the Aubry sets. 
\end{abstract}

\section{Introduction}

Let $H(x, p)$ be a Tonelli Hamiltonian on $\T^n \times \R^n$, the Aubry set $\tcA \subset \T^n \times \R^n$ is one of the fundamental variationally defined invariant sets. Arnaud (see for example \cite{Arn2010}, \cite{Arn2011}, \cite{Arn2012}, \cite{Arn2016}) developed a theory linking the regularity of the Aubry set to the Green bundles. 

The Green bundles $\cG_\pm(x, p) \subset \R^n \times \R^n$ is a family of invariant Lagrangian subspaces transversal to the vertical $\{0\} \times \R^n$, which means they are given by the graph of symmetric matrices: $\cG_\pm = \{(h, G_\pm h):\, h \in \R^n\}$. Given two such Lagrangian subspaces  $\cS_i = \{(h, S_i h)\}$, $i =1, 2$, we say $\cS_2 > \cS_1$ if $S_2 > S_1$, meaning $S_2 - S_1$ is positive definite. Then $\cG_- \le \cG_+$. 
We will also consider the \emph{modified Green bundle} (Arnaud, \cite{Arn2010}) $\widetilde{\cG}_-$ and $\widetilde{\cG}_+$, defined by the matrices $G_+ + (G_+ - G_-)$ and $G_- - (G_+ - G_-)$. Clearly $\widetilde{\cG}_- \le \cG_- \le \cG_+ \le \widetilde{\cG}_+$. 

Let $\cS_i = \{(h, S_i h)\}$, $i =1, 2$ be such that $S_1 \le S_2$, We define the \emph{cone} between $\cS_1, \cS_2$ as:
\[
	C(\cS_1, \cS_2) = \{ (h, Sh): \, S_1 \le S \le S_2, \quad h \in \R^n\}. 
\]

We will consider the following different definitions of tangent cones. 
\begin{itemize}
 \item  The \emph{contingent cone} $\cC_z(\tcA)$ of the  set $\tcA$ at $z \in \tcA$, is defined as the set of all limit points $\lim_{n \to \infty} t_n(z_n - z)$ for $t_n > 0$ , $z_n \in \tcA$ and $z_n \to z$.
 \item The \emph{limit contingent cone} $\widetilde{\cC}_z(\tcA)$ is the set of all limit points of vectors $v_n \in \cC_{z_n}(\tcA)$ with $z_n \in \tcA$ and $z_n \to z$. 
 \item The \emph{paratingent cone} $\cP_x(\tcA)$ is defined as the limit points of $\lim_{n \to \infty} t_n(z_n - w_n)$, where $t_n > 0 $, $z_n, w_n \in \tcA$ and $z_n , w_n \to z$. 
\end{itemize}
Clearly we have $\cC_z(\tcA) \subset \widetilde\cC_z (\tcA) \subset \cP_z(\tcA)$. 
The following result is due to Arnaud:
\begin{thm}[\cite{Arn2011}, \cite{Arn2012}]
	In the case of Tonelli Hamiltonian, we have $\widetilde\cC_z(\tcA) \subset C(\widetilde{\cG}_-(z), \widetilde{\cG}_+(z))$. In the  case of a twist map on the space $\T \times \R$, the result improves to $\cP_z(\tcA) \subset C(\cG_-(z), \cG_+(z))$. 
\end{thm}
Arnaud asks in \cite{Arn2016}, Question 6, 7, whether the two improvements (limit contingent cone to paratingent cone, and modified Green bundle to original Green bundle) are possible for general Tonelli Hamiltonians. We answer both questions positively.
\begin{thm}\label{thm:main}
	For the Tonelli Hamiltonians, we have $\cP_z(\tcA) \subset C(\cG_-(z), \cG_+(z))$ for all $z \in \tcA$. 
\end{thm}

Arnaud also discovered the relation between Green bundles and the Lyapunov exponents of minimal measures. Among other results, she proved that (\cite{Arn2012}) if a minimal measure has only zero exponents, then on the support of the minimal measure, the Aubry set is \emph{$C^1$-isotropic}, meaning $\widetilde\cC_z(\tcA)$ is contained in a Lagrangian subspace. As mentioned in \cite{Arn2016}, our result improves this regularity to \emph{$C^1$-regular}, meaning $\cP_z(\tcA)$ is contained in a Lagrangian subspace. 

We prove our result by first giving an alternative characterization of the symplectic cone, see Section~\ref{sec:symp-con}. We then develop an anisotropic version of the standard semi-concavity, and use it to derive an upper bound for the paratingent cone, see Section~\ref{sec:semi-concave}. Finally in Section~\ref{sec:green-bundle}, we show that the weak KAM solutions satisfy the new semi-concavity conditions, and use it to prove our main theorem. 

\section{Characterization of the symplectic cone}
\label{sec:symp-con}

A subset $K \subset \R^{2n}$ is called a cone if $0 \in K$ and $\lambda K \subset K$ for all $\lambda > 0$. Under our definition, a cone is uniquely determined by its intersection with the unit sphere. The space of all non-trivial closed cones then form a complete metric space using the Hausdorff topology on the unit sphere. In particular, this also induces a metric on the space of non-zero subspaces. 

We give an alternative characterization for the cone $C(\cG_-, \cG_+)$. 
Let $\cL_1, \cL_2 \subset \R^{2n}$ be Lagrangian subspaces. Define a function $\Sg: 
\R^{2n} \to \R \cup \{-\infty\}$ by 
\[
	\Sg_{\cL_1, \cL_2}(v) = \omega(v_1, v_2), \quad v = v_1 + v_2, \quad v_1 \in \cL_1, v_2 \in \cL_2, 
\]
and $\Sg_{\cL_1, \cL_2}(v) = - \infty$ if $v \notin \cL_1+ \cL_2$. 
\begin{lem}
	The function $\Sg(v)$ is well defined.
\end{lem}
\begin{proof}
	Suppose $u \in \cL_1 \cap \cL_2$, then 
	\[
		\omega(v_1 + u, v_2 - u) = \omega(v_1, v_2) + \omega(v_1, -u) + \omega(u, v_2) + \omega(u, -u) = \omega(v_1, v_2)
	\]
	since $\cL_1, \cL_2$ are Lagrangian. Therefore $\Sg(v)$ is well defined. 
\end{proof}

We have the following characterization: 
\begin{prop}\label{prop:cone-sign}
Let $\cS_1 \le \cS_2$. Then 
\[
	C(\cS_1, \cS_2) = \{v: \quad \Sg_{\cS_1, \cS_2}(v) \ge 0\}. 
\]
\end{prop}

We prove this proposition in two steps. First we assume the subspaces $\cS_1, \cS_2$ are transversal. 

\subsection{The transversal case}

\begin{lem}\label{lem:matrix}
Suppose $y_1, y_2 \in \R^n$ satisfies $y_1^T y_2 \ge 0$, then there exists $W_1, W_2$ positive semi-definite symmetric matrices, such that 
\[
	y_1  = W_1 (y_1 + y_2), \quad y_2 = W_2 (y_1 + y_2), \quad W_1 + W_2 = \Id. 
\]
\end{lem}
\begin{proof}
	Let $y = y_1 + y_2$, and $z = y_1 - y_2$, then 
	\[
		y_1^T y_2 \ge 0 \quad \Leftrightarrow \quad \|z\| \le \|y\|. 
	\]
	We show the following: there exists a  symmetric matrix $W$ such that $- \Id \le W \le \Id$ and 
	\[
		z = W y. 
	\]
	
	First, by scaling both $z$ and $y$, it suffices to consider $\|y\|= 1$ and $\|z\|\le 1$.  Let $P$ be an orthogonal matrix such that $Py = e_1$, $P z \in \Span\{e_1, e_2\}$. Then $Pz = (z_1, z_2, 0, \cdots, 0)$ with $z_1^2 + z_2^2 \le 1$. Define
	\[
		W' = 
		\begin{bmatrix}
			\begin{matrix}
				z_1 & z_2 \\ z_2 & - z_1
			\end{matrix} & 0 \\
			0 & 0 
		\end{bmatrix}, 
	\]
	then $- \Id \le W' \le \Id$ and $W' e_1 = P z$. As a result, $W = P^T (W') P$ satisfies $-\Id \le W \le \Id$ and $W y = z$. 

	We now let 
	\[
		W_1 = \frac12(\Id + W), \quad W_2 = \frac12 (\Id - W), 
	\]
	then 
	\[
		W_1 y = (y + z)/2 = y_1, \quad W_2 y = (y-z)/2 = y_2. 
	\]
\end{proof}

\begin{lem}\label{lem:char-cone}
Suppose $\cS_1 < \cS_2$, then
\begin{equation}
  \label{eq:char-cone}
  	C(\cS_1, \cS_2) = \{v: \quad \Sg_{\cS_1, \cS_2}(v) \ge 0\}. 
\end{equation}
\end{lem}
\begin{proof}
	\emph{Part 1.}
	Let us first show $v = (x, S x)$ with $S_1 \le S \le S_2$ implies $\Sg(v) \ge 0$.  Then $v = v_1 + v_2$, $v_1 = (x_1, S_1 x_1) \in \cS_1$, $v_2 = (x_2, S_2 x_2) \in \cS_2$ if and only if 
	\[
		\bmat{\Id & \Id \\ S_1 & S_2} \bmat{x_1 \\ x_2} = \bmat{x \\ S x}. 
	\]
	Denote $U = (S_2 - S_1)$, we have 
	\begin{equation}
	  \label{eq:decomp}
	  		\bmat{x_1 \\ x_2} = \bmat{U^{-1} & 0 \\ 0 & U^{-1}} \bmat{S_2 & - \Id \\ - S_1 & \Id} \bmat{x \\ S x} = 
		\bmat{ U^{-1}(S_2 - S) x \\ U^{-1} (S - S_1)x }. 
	\end{equation}
	Noting that 
	\[
		\Sg(v) = \omega( (x_1, S_1 x_1), (x_2, S_2 x_2) ) = x_1^T S_2 x_2 - x_1^T S_1^T x_2 =  x_1^T U x_2, 
	\]
	we obtain 
	\[
		\Sg(v) = x^T (S_2 - S) U^{-1} (S - S_1) x.
	\]
	It suffices to show that $(S_2 - S) U^{-1} (S - S_1)$ is positive semi-definite. Denote $U_2 = S_2 - S$ and $U_1 = S - S_1$, then both $U_1, U_2$ are positive semi-definite, and $U_1 + U_2 = U$. We note that 
	\[
		U_2 U^{-1} U_1 = (U - U_1)U^{-1}(U - U_2) = U - U_1 - U_2 + U_1 U^{-1} U_2 = U_1 U^{-1} U_2,
	\]
	therefore $U_2 U^{-1} U_1$ is symmetric. Furthermore, since 
	\[
		U^{-\frac12} \left( U_2 U^{-1} U_1 \right) U^{-\frac12} = (U^{-\frac12} U_2 U^{-\frac12})(U^{-\frac12} U_1 U^{-\frac12})
	\]
	is the product of two commuting positive semi-definite symmetric matrices, it is itself positive semi-definite. 

	\emph{Part 2}. 
	For the converse, let $\Sg(v) \ge 0$. Define 
	\begin{equation}
	  \label{eq:decomp2}
	  		\bmat{x_1 \\ x_2} = \bmat{U^{-1} & 0  \\ 0 & U^{-1}} \bmat{S_2 & - \Id \\ - S_1 & \Id} v,
	\end{equation}
	then $v = (x_1 + x_2, S_1 x_1 + S_2 x_2)$, and 
	\[
		\Sg(v) = x_1^T U x_2. 
	\]
	Let $y_i = U^{\frac12} x_i$ for $i = 1, 2$, then $\Sg(v) \ge 0$ implies $y_1^T  y_2 \ge 0$. 

	Use Lemma~\ref{lem:matrix}, we get $W_1, W_2 \ge 0$, $W_1 + W_2 = \Id$, such that 
	\[
		W_2 (y_1 + y_2) = y_1, \quad W_1 (y_1 + y_2) = y_2. 
	\]
	Since $y_i = U^{\frac12} x_i$, denote
	\[
		U_1 = U^{\frac12} W_1 U^{\frac12}, \quad U_2 = U^{\frac12} W_2 U^{\frac12}, 
	\]
	we get 
	\[
		U_2 (x_1 + x_2) =  U x_1, \quad U_1 (x_1 + x_2) = U x_2, \quad U_1, U_2 \ge 0, \quad U_1 + U_2 = U. 
	\]
	Write $x = x_1 + x_2$, we have 
	\[
\begin{aligned}
 		v & = (x_1 + x_2, S_1 x_1 + S_2 x_2) = (x_1 + x_2, S_1 (x_1 + x_2) + U x_2) \\
& = (x, S_1 x + U_1 x) = (x, (S_1 + U_1)x). 
 \end{aligned}
	\]
	Therefore  $v = (x, S x)$ where $S = U_1 + S_1 = S_2 - U_2$ satisfies $S_1 \le S \le S_2$. 
\end{proof}

\subsection{The general case}

We reduce the general case to the transversal case by using a coordinate change. Let $C \in \R$ and $A$ an invertible $n \times n$ matrix, consider the following linear symplectic maps $\Phi_C, \Phi_A: \R^{2n} \to \R^{2n }$
\[
	\Phi_C(x,y) = (x, y + C x),  \quad \Phi_A(x, y) = \left(  A^{-1} x, A^T y \right). 
\]
Let $S$ be a symmetric matrix and $\cS$ the associated Lagrangian subspace. Denote $\cS^A$ the Lagrangian subspace defined by the symmetric matrix $A^T S A$. 
\begin{lem}
\label{lem:symp-inv} Let $S_1 \le S_2$ be symmetric matrices, $C \in \R$ and $A$ an invertible matrix.
\begin{enumerate}
 \item  The equality \eqref{eq:char-cone} holds for $S_1, S_2$ if and only if the same holds for $S_1 + C \Id, S_2 + C \Id$. 
 \item  The equality \eqref{eq:char-cone} holds for  $S_1, S_2$ if and only if the same  holds for $A^T S_1 A \le A^T S_2 A$. 
\end{enumerate}
\end{lem}
\begin{proof}
Since the symplectic form $\omega$ is invariant under symplectic maps, for any linear symplectic map $\Phi$ and Lagrangian subspaces $\cL_1, \cL_2$, 
\[
	\Phi\left( \{v: \, \Sg_{\cL_1, \cL_2}\}(v) \ge 0\} \right) = \{v: \, \Sg_{\Phi \cL_1, \Phi \cL_2}(v) \ge 0 \}. 
\]

For (1), let us denote by $\cS_1', \cS_2'$ the subspaces for $S_1 + C \Id, S_2 + C \Id$. Then we clearly have 
\[
	\Phi_C \cS_i = \cS_i',  \, i = 1, 2, \quad \Phi_C\left(  C(\cS_1, \cS_2)  \right)= C(\cS_1', \cS_2'),
\]
(1) follows. 

For (2), let us denote by $\cS_i^A$ the subspaces of $A^T S_i A$, $i = 1,2$. Then 
\[
	\Phi_A \cS_i = \{(A^{-1} x, A^T S_ix): \,  x\in \R^n\} = \{ (A^{-1} x, A^T S_i A A^{-1} x): \, x \in \R^n \} = \cS^A_i.  
\]
On the other hand, since $S_1 \le S_2$ if and only if $A^T S_1 A \le A^T S_2 A$, we have
\[
\begin{aligned}
    	& \Phi_A \left( C(\cS_1, \cS_2) \right) = \Phi_A \{ (x, Sx): \, x \in \R^n, \, S_1 \le S \le S_2 \}  \\
    	& \quad = \{ (A^{-1} x, A^T S A A ^{-1} x):\,  x \in \R^n, \, S_1 \le S \le S_2\}     = C(\cS_1^A, \cS_2^A). 
\end{aligned}
\]
Item (2) follows. 
\end{proof}

\begin{proof}
[Proof of Proposition~\ref{prop:cone-sign}]
It suffices to consider the case when $\cS_1, \cS_2$ are not transversal. Moreover, by applying the symplectic coordinate change $\Phi_C$ for $C$ sufficiently large, we may assume that $S_1, S_2$ are both invertible. 

Assume that $S_1 \le S_2$ and $\dim \ker (S_2 - S_1) = n-m > 0$. Let $P$ be an orthogonal matrix which maps $\ker(S_2 - S_1)$ to the subspace $\{0\} \times \R^{n - m}$ and maps $\left( \ker (S_2 - S_1) \right)^\perp$ to $\R^m \times \{0\}$. Since $\ker P(S_2 - S_1) P^T = \{0\} \times \R^{n-m}$, in block form we have
\[
	PS_1 P^T = \bmat{\tilde{S}_1 & M \\ M^T & N }, \quad P S_2 P^T = \bmat{\tilde{S}_2 & M  \\ M^T & N},  \quad
	\det(\tilde{S}_2- \tilde{S}_1) \ne 0, \quad \det N \ne 0. 
\]
Consider the matrix 
\[
	Q = \bmat{ I & 0 \\ - N^{-1} M^T & I}, 
\]
then for $i = 1,2$, 
\[
\begin{aligned}
  & 	Q^T P S_i P^T Q \\
  & \quad = 	\bmat{I & - M N^{-1}  \\ 0 & I} \bmat{\tilde{S}_i & M \\ M^T & N } \bmat{I & 0 \\ - N^{-1} M^T & I}
	=  \bmat{\tilde{S}_i - M N^{-1} M^T & 0 \\ 0 & N}. 
\end{aligned}
\]
Therefore, by considering the coordinate change $\Phi_A$ where $A = P^T Q$, we reduce to the special case
\[
	S_1 = \bmat{\bar{S}_ 1& 0 \\ 0 & N}, S_2 = \bmat{\bar{S}_2 & 0 \\ 0 & N}, 
\]
where $\bar{S}_1 <  \bar{S}_2$. In this special case
\[
	C(\cS_1, \cS_2) = \left\{ \left( \bmat{\bar{x} \\ \bar{y}}, \bmat{\bar{S} \bar{x} \\ N \bar{y}} \right): \quad \bar{S}_1 \le \bar{S} \le \bar{S}_2, \, \bar{x} \in \R^m, \, \bar{y} \in \R^{n-m}  \right\}.  
\]
On the other hand, every $v \in \cS_1 + \cS_2$ can be expressed as 
\[
\begin{aligned}
   	v &= \left( \bmat{\bar{x}_1 + \bar{x}_2 \\ \bar{y}_1 + \bar{y}_2} , 
	\bmat{\bar{S}_1 \bar{x}_1 + \bar{S}_2 \bar{y}_2 \\ N\bar{y}_1 + N \bar{y}_2} \right) \\
	& = w_1 + w_2 := \left( \bmat{\bar{x}_1 + \bar{x}_2 \\ 0} , \bmat{\bar{S}_1 \bar{x}_1 + \bar{S}_2 \bar{y}_2 \\ 0} \right) + 
	\left( \bmat{0 \\ \bar{y}}, \bmat{0 \\ N \bar{y}} \right), 
\end{aligned}
\]
where $\bar{y} = \bar{y}_1 + \bar{y}_2$. Since $w_2\in \cS_1 \cap \cS_2$, $\Sg_{\cS_1, \cS_2}(v) = \Sg_{\cS_1, \cS_2}(w_1)$. Our proposition now follows from applying Lemma~\ref{lem:char-cone} to the reduced matrices $\bar{S}_1, \bar{S}_2$. 
\end{proof}

\section{Generalized semi-concavity and tangent cones}
\label{sec:semi-concave}

Let $\Omega \subset \R^n$ be an open convex set. 
A function $f: \Omega \to \R$ is called $C$-semi-concave if for each $x \in \Omega$, there is $l_x \in \R^n$ such that 
\[
	f(y) - f(x) -  l_x \cdot (y - x) \le \frac12 C \|y - x\|^2, \quad x, y \in \Omega. 
\]
$l_x$ is called a super-gradient at $x$.
$f$ is called $C$-semi-convex if $-f$ is $C$-semi-concave, and $l_x$ is called a sub-gradient.   It is well known if a function is both semi-concave and semi-convex, then $f$ is differentiable, and $df$ is locally Lipschitz. In this section we outline a generalized version of this lemma. 

Let $A$ be a symmetric  $n\times n$ matrix. We say that $f: \Omega \to \R$ is $A$-semi-concave if for each $x \in \R^n$, there is $l_x \in \R^n$ such that 
\[
	f(y) - f(x) -  l_x \cdot (y - x) \le \frac12 A (y-x)^2, \quad x, y \in \Omega
\]
where $Ax^2$ denotes $Ax \cdot x$. $A$-semi-concave functions are $\|A\|$-semi-concave.  We say $f$ is $A$-semi-convex if $-f$ is $A$-semi-concave. The following lemma is proven by direct computation.
\begin{lem}
	$f$ is $A$-semi-concave if and only if $f_A(x) = f(x) - \frac12 Ax^2$ is concave. 
\end{lem}

The proof of our  next lemma follows Proposition 13.33 in \cite{RW2009}. 
\begin{lem}\label{lem:aniso-lipshitz}
Suppose $f: \R^n \to \R$ is $B$-semi-concave and $g: \R^n \to \R$ is $(-A)$-semi-convex, and $U = B - A$ is positive definite. Suppose $f(x)\ge g(x)$ and $K$ is the set on which $f - g$ reaches its minimum. 

Then for all $x_1, x_2 \in K$, we have 
\[
	\|df(x_2) - df(x_1) - \frac12(A + B)(x_2 - x_1)\|_{U^{-1}} \le \frac12 \|x_2 - x_1\|_U, 
\]
where $\|x\|_U = \sqrt{U x^2}$. 

The same conclusion holds, if $f, g$ are only defined on $\Omega$, and we assume in addition that $K = \argmin(f-g)$ is compactly contained in $\Omega$. 
\end{lem}
\begin{proof}
	First of all, by adding a constant to $g$, we may assume that $\min(f - g) = 0$. Then by standard estimates of semi-concave functions, we have $f(x) = g(x)$ and $df(x) = dg(x)$ on $K$. 

	Since $g$ is $-A$-semi-convex,  $g_{A}(x) = g(x) - \frac12 A x^2$ is convex, and $dg_A(x) = dg(x) - Ax$ where $g$ is differentiable. Since $g_A$ is convex, convex duality (see for example \cite{RW2009}, chapter 11) implies if $p_1^A$ is any sub-gradient of $g_A$ at $x_1$, we have 
	\[
		g^*_A(p_1^A) := \sup_x \left\{ p_1^A \cdot x - g_A(x) \right\} = p_1^A \cdot x_1 - g(x_1). 
	\]
	We note that $g_A(x) \le f_A(x)$, with equality holding on $K$.  $f_A(x)$ is $U$-semi-concave. 

	Let $x_1, x_2 \in K$, then $f(x_i) = g(x_i)$ and $p_i = df(x_i) = dg(x_i)$, $p_i^A = p_i - A x_i$, $i = 1, 2$,  then 
	\[
		\begin{aligned}
			 g_A^*(p_2^A) &= \sup_x \{ p_2^A\cdot x  - g_A(x)\} \ge \sup_x \{ p_2^A\cdot x  - f_A(x)\}  \\
			& \ge \sup_x \{  p_2^A \cdot x - f_A(x_1) -  p_1^A \cdot  (x - x_1) -  \frac12 U(x - x_1)^2  \} \\
			& = p_1^A\cdot x_1 - g_A(x_1) + \sup_x \{ (p_2^A - p_1^A) \cdot x - \frac12 U(x - x_1)^2  \} \\
			& = g_A^*(p_1^A) + (p_2^A - p_1^A) \cdot x_1 + \frac12 U^{-1}(p_2^A - p_1^A)^2, 
		\end{aligned}
	\]
	where in the last equality, we used the fact that $\sup_x \{p \cdot x - \frac12 U x^2\} = \frac12 U^{-1} p^2$. Switch $x_1$ and $x_2$, we obtain 
	\[
		g_A^*(p_1^A) \ge g_A^*(p_2^A )+ (p_1^A - p_2^A) \cdot x_2 + \frac12 U^{-1} (p_2^A - p_1^A)^2. 
	\]
	Sum the two inequalities obtained, we have 
	\[
		U^{-1}(p_2^A - p_1^A)^2 \le (p_2^A - p_1^A) \cdot (x_2 - x_1) =  U^{-1}(p_2^A - p_1^A) \cdot U(x_2 - x_1). 
	\]
	Complete squares, we get 
	\[
		U^{-1}\left( p_2^A - p_1^A - \frac12 U(x_2 - x_1) \right)^2 \le \frac14 U(x_2 - x_1)^2,
	\]
	and the left hand side is equal to $U^{-1}\left( p_2 - p_1 - \frac12 (A + B)(x_2 - x_1) \right)^2$.

	For the local version, we only need to extend both $f, g$ to  $\R^n$ keeping the same semi-concavity, and that on $f - g > 0$ on $\R^n \setminus \Omega$. 
\end{proof}

We obtain the following standard lemma due to Fathi (see \cite{Fat2011}, \cite{Arn2012}) as a corollary. 
\begin{cor}\label{cor:graph-theorem}
If $f, -g$ are $C$-semi-concave, and $f \ge g$, $K = \argmin(f-g)$, then there is $C'> 0$ such that 
\[
	\|df(x_2) - df(x_1)\| \le C' \|x_2 - x_1\|, \quad x_1, x_2 \in K. 
\]
\end{cor}

Under the same assumptions as Lemma~\ref{lem:aniso-lipshitz}, define 
\begin{equation}
	\label{eq:cI}
	\cI_{f, g} = \argmin (f - g), \quad \tcI_{f, g} = \{(x, df(x)): x \in \cI_{f, g}\}, 
\end{equation}
and write $\cS_A = \{(h, Ah)\}$, $\cS_B = \{(h, Bh)\}$, then:
\begin{cor}\label{cor:para-cone}
	Under the same assumptions as Lemma~\ref{lem:aniso-lipshitz}, for every $z = (x, df(x)) \in \tcI_{f, g}$,
	\[
		\cP_z \tcI_{f, g} \subset \cC(\cS_A, \cS_B), 
	\]
	where $\cP_z$ is the paratingent cone. 
\end{cor}
\begin{proof}
	Consider for $i =1, 2$ and $n \in \N$,   $(x_i^n, p_i^n) \in \tcI$, $(x_i^n, p_i^n) \to z$, $t_n > 0$, and $(t_n(x_2^n - x_1^n), t_n(p_2^n - p_1^n)) \to (h, k) \in \R^n \times \R^n$. Lemma~\ref{lem:aniso-lipshitz} implies
	\[
		\left\| p_2^n - p_1^n - \frac12(A + B)(x_2^n - x_1^n) \right\|_{U^{-1}} \le \frac12 \|x_2^n - x_1^n\|_U. 
	\]
	Multiply by $t_n$ and take limit, we get 
	\[
		\cP_z \tcI_{f, g} \subset \left\{(h, k): \quad \|k  - \frac12 (A + B) h\|_{U^{-1}} \le \frac12 \|h\|_U \right\} , 
	\]
 	it suffices to show the right hand side is equal to $\cC(\cS_A, \cS_B)$.

	We apply \eqref{eq:decomp2} with $S_2 = B$, $S_1 = A$ and $U = B - A$, then 
	\[
		(h , k) = \left( h_1 + h_2, Ah_1 + Bh_2  \right), \quad h_1 = -U^{-1} k + U^{-1} B h, \quad h_2 = U^{-1} k - U^{-1} A h	. 
	\]
		Denote $q = k - \frac12(A + B)h$, 
	\[
		\begin{aligned}
			& 	\Sg_{\cS_A, \cS_B}( (h, k) ) = Uh_1 \cdot h_2 = (-k + B h) \cdot U^{-1}(k - A h) \\
			& = \left(  -q + \frac12 U h \right) \cdot U^{-1} \left( q + \frac12 U h \right) =  
			- \|q\|_{U^{-1}}^2 + \frac14 \|h\|_U^2, 
		\end{aligned}
	\]
	therefore 
	\[
		\left\| k  - \frac12 (A + B) h \right\|^2_{U^{-1}} \le \frac14 \|h\|_U^2 
		\quad \Leftrightarrow
		\quad \Sg_{\cS_A, \cS_B}( (h,k) ) \ge 0
	\]
	which is exactly what we need in view of Proposition~\ref{prop:cone-sign}. 
\end{proof}

\section{The Aubry set and the Green bundles}
\label{sec:green-bundle}

Let $L$ denote the Lagrangian associated to $H$. The action function is
\[
	A^t(x, y) = \inf \left\{ \int_0^t L(\gamma, \dot{\gamma}) dt: \quad \gamma(0)= x, \, \gamma(t) = y \right\}. 
\]
The (backward) Lax-Oleinik semi-group $T_t: C(\T^n) \to C(\T^n)$ is defined as 
\[
	T_t u(y) = \min_{x \in \T^n} \{ u(x) + A^t(x, y) \}, 
\]
and $u: \T^n \to \T$ is called a weak KAM solution if $T_t u = u$. The forward semi-group is
\[
	T^+_t u(x) = \max_{y \in \T^n} \{ u(y) - A^t(x, y)\}. 
\]
$u$ is called a weak KAM solution if there is $c \in \R$ such that  $T_t u + ct = u$. Similarly, $w$ is called a forward weak KAM solution if $T^+_t w - ct = w$. We refer to \cite{Fat2011} for a wealth of information on weak KAM theory. 

The Mather set $\tcM$ is the support of all minimal invariant probabilities to the Euler-Lagrange flow, namely, ones that minimizes $\int L(x, v) d\mu(x, v)$. The projected Mather set $\cM$ is its projection to $\T^n$. Fathi (\cite{Fat2011}) showed that given any weak KAM solution $u$, there is a unique forward solution $w \le u$ such that $u = w$ on $\cM$. The pair $(u, w)$ is called a \emph{conjugate pair}. 

Let $(u, w)$ be a conjugate pair, $\cI_{u, w}$ and $\tcI_{u, w}$ as in \eqref{eq:cI}, we define the Aubry set 
\[ 
	\tcA = \bigcap \{ \tcI_{u, w}:\quad 	(u, w) \text{ is a conjugate pair }
	\}. 
\]
Each $\tcI_{u, w}$ is contained in a  Lipschitz graph with a uniform Lipschitz constant due to Corollary~\ref{cor:graph-theorem}. Each $\tcI_{u, w}$, and therefore $\tcA$, is a compact invariant set of the Hamiltonian flow.

An orbit $z(t) = (x, p)(t)$ is called \emph{disconjugate}  if for all $t_1, t_2 \in \R$, we have 
\[
	D \phi_{t_1}^{t_2}\,  \bV(x(t_1), p(t_1)) \pitchfork \bV(x(t_2), p(t_2)), 
\]
where $\bV(x, p) = \{0\} \times \R^n \subset T_{(x, p)}(\T^n \times \R^n)$ is called the vertical subspace. Every orbit in the set $\tcI_{u, w}$ is disconjugate. Given a disconjugate orbit, we define the pre-Green bundles
\[
	\cG_{t}(z) = D\phi_{t} \, \bV(\phi_{-t}z), \quad \cG_{-t}(z) = \left( D\phi_t \right)^{-1} \bV(\phi_t z). 
\]
$\cG_t(z)$ are Lagrangian subspaces given by symmetric matrices $G_t(z)$. 
\begin{prop}[See for example \cite{Gre58}, \cite{CI1999}, \cite{Itu2002}, \cite{BM2004}]
	For all $s, t>0$, $G_{-s} > G_t$, and  $G_{-t}$ is decreasing in $t > 0$ and $G_t$ increasing in $t >0$. As a result
	\[
		\cG_+(z) = \lim_{t \to \infty} \cG_{t}(z) , \quad 
		\cG_-(z) = \lim_{t \to \infty} \cG_{-t}(z)
	\]
	are invariant subbundles along the orbit of $z$. 
\end{prop}

\begin{prop}[\cite{Arn2011}] \label{prop:act-green-bundle}
Suppose $\gamma: \R \to \T^n $ is a minimizing orbit. Then for each $t-s = T > 0$, the function $A^T(x, y)$ is a $C^2$ function in a neighborhood of $(\gamma(s), \gamma(t))$. Moreover, we have 
\[
	G_{T}(\gamma(t)) = \partial^2_{22} A^T(\gamma(s), \gamma(t)), \quad
	 G_{-T}(\gamma(s)) =  - \partial^2_{11} A^T(\gamma(s), \gamma(t))
\]
\end{prop}

\begin{lem} \label{lem:local-semi-concave}
Let $(u, w)$ be a conjugate pair of weak KAM solutions, and let $\gamma(t)$ be the projection of an orbit in $\tcI_{u, w}$. Then for each $\epsilon > 0$,  there is a neighborhood $V$ of $x_0 = \gamma(0)$ on which
\begin{itemize}
 \item  $u$ is $(G_{T}(x_0) + \epsilon \Id)$-semi-concave;
 \item  $w$ is $ -(G_{-T}(x_0) - \epsilon \Id)$-semi-convex.  
\end{itemize}
\end{lem}
\begin{proof}
By Proposition~\ref{prop:act-green-bundle}, the functions $A^T(x, y)$ is $C^2$ near both $(\gamma(-T), x_0)$ and $(x_0, \gamma(T))$. Using the relation of second derivatives in Proposition~\ref{prop:act-green-bundle},  for any $\epsilon > 0$, there is $\delta > 0$ such that for all $x \in B_\delta(\gamma(-T))$, the function  of $A^T(x, \cdot)$ is $G_{-T}(x_0) + \epsilon \Id$ semi-concave on $B_\delta(x_0)$,  here $B_\delta(x)$ denote ball of radius $\delta$ at $x$. 

Let $y_1, y_2 \in B_{\delta'}(\gamma(-T))$, where $\delta' < \delta$ is to chosen, then there exists minimizing curves $\gamma_1, \gamma_2 :(-\infty, 0] \to \T^n$  (called calibrated curves, see \cite{Fat2011}) such that $\gamma_i(0) = y_i$, $i = 1, 2$ and $u(y_i) = u(\gamma_i(-t)) +  A^t(\gamma_i(-t), y_i)$. By choosing $\delta'$ small, we can assume $\gamma_i(-T) \in B_\delta(\gamma(-T))$, and as a result
\[
\begin{aligned}
 	u(y_2) - u(y_1) & = u(y_2) - u(\gamma_1(-T)) + A^T(\gamma_1(-T), y_1) \\
& \le A^T(\gamma_1(-T), y_2) - A^T(\gamma_1(-T), y_1) \\
&  \le l_{y_1}(y_2 - y_1) + \frac12 \left( G_{T}(x_0) + \epsilon \Id \right) (y_2 - y_1)^2
\end{aligned}
\]
by semi-concavity of $A^T(\gamma_1(-T), \cdot)$. The proof for semi-convexity of $w$ is similar. 
\end{proof}

\begin{proof}[Proof of Theorem~\ref{thm:main}]
Let $\cG_{\pm T}^\epsilon$ denote the Lagrangian subspaces associated to $ G_{\mp T}(z_0) \pm \epsilon \Id$, then Lemma~\ref{lem:local-semi-concave}, together with Corollary~\ref{cor:para-cone} implies 
\[
	\cP_{z_0} \tcI_{u, w} \subset \cC(\cG_{-T}^\epsilon, \cG_{T}^\epsilon). 
\]
Take $T \to \infty$, we get 
\[
	\cP_{z_0} \tcI_{u, w} \subset \cC(\cG_-^\epsilon, \cG_+^\epsilon),  
\]
where $\cG_\pm^\epsilon$ are defined by $ G_\pm \pm \epsilon \Id$. Finally, we get our conclusion by taking intersection over all $\epsilon$. 
\end{proof}

\subsection*{Acknowledgment}
The author thanks the referee for suggesting using reduction for degenerate case of the proof of Proposition 2.2, and for many corrections leading to improvements of the paper. He also thanks Vadim Kaloshin and Alfonso Sorrentino for discussions leading to this paper, and Marie-Claude Arnaud on helpful comments on a draft version. K.Z. is supported by the NSERC Discovery grant, reference number 436169-2013.

\bibliographystyle{abbrv}
\bibliography{green-bundles}

\begin{thebibliography}{10}

\bibitem{Arn2010}
M.-C. Arnaud.
\newblock Green bundles and related topics.
\newblock In {\em Proceedings of the {I}nternational {C}ongress of
  {M}athematicians. {V}olume {III}}, pages 1653--1679. Hindustan Book Agency,
  New Delhi, 2010.

\bibitem{Arn2011}
M.-C. Arnaud.
\newblock The link between the shape of the irrational {A}ubry-{M}ather sets
  and their {L}yapunov exponents.
\newblock {\em Ann. of Math. (2)}, 174(3):1571--1601, 2011.

\bibitem{Arn2012}
M.-C. Arnaud.
\newblock Green bundles, lyapunov exponents and regularity along the supports
  of the minimizing measures.
\newblock In {\em Annales de l'Institut Henri Poincare (C) Non Linear
  Analysis}, volume~29, pages 989--1007. Elsevier, 2012.

\bibitem{Arn2016}
M.-C. Arnaud.
\newblock Lyapunov exponents for conservative twisting dynamics: a survey.
\newblock In {\em Ergodic theory}, pages 108--133. De Gruyter, Berlin, 2016.

\bibitem{BM2004}
M.~L. Bialy and R.~S. MacKay.
\newblock Symplectic twist maps without conjugate points.
\newblock {\em Israel Journal of Mathematics}, 141:235--247, 2004.

\bibitem{CI1999}
G.~Contreras and R.~Iturriaga.
\newblock Convex hamiltonians without conjugate points.
\newblock {\em Ergodic Theory and Dynamical Systems}, 19(04):901--952, 1999.

\bibitem{Fat2011}
A.~Fathi.
\newblock Weak {KAM} theorem in lagrangian dynamics, book preprint, 2011.

\bibitem{Gre58}
L.~W. Green.
\newblock A theorem of {E}. {H}opf.
\newblock {\em Michigan Math. J.}, 5:31--34, 1958.

\bibitem{Itu2002}
R.~Iturriaga.
\newblock A geometric proof of the existence of the green bundles.
\newblock {\em Proceedings-american Mathematical Society}, 130(8):2311--2312,
  2002.

\bibitem{RW2009}
R.~T. Rockafellar and R.~J.-B. Wets.
\newblock {\em Variational analysis}, volume 317.
\newblock Springer Science \& Business Media, 2009.

\end{thebibliography}

\end{document}